\documentclass[a4paper, 12pt]{amsart}
\usepackage[cp1251]{inputenc}
\usepackage[english]{babel}
\usepackage{amsmath,amssymb,a4wide,latexsym,amscd,amsthm}
\usepackage{graphics}

\theoremstyle{plain}
\newtheorem{theorem}{Theorem}[section]

\newtheorem{lemma}{Lemma}[section]
\newtheorem{proposition}{Proposition}[section]
\newtheorem{corollary}{Corollary}[section]

\theoremstyle{definition}
\newtheorem{example}{Example}[section]

\newtheorem{remark}{Remark}

\numberwithin{equation}{section}

\begin{document}

\title[]
{On absolutely representing families of subspaces in Banach spaces}

\author{Ivan S. Feshchenko}

\maketitle

\begin{abstract}
Let $X$ be a Banach space, and $X_\lambda$, $\lambda\in\Lambda$, be a family of subspaces of $X$.
This family is called an absolutely representing family of subspaces in $X$ if for every $x\in X$ there exist
$x_\lambda\in X_\lambda$, $\lambda\in\Lambda$, such that $\sum_{\lambda\in\Lambda}\|x_\lambda\|<\infty$ and $x=\sum_{\lambda\in\Lambda}x_\lambda$.

In this paper we obtain necessary and (or) sufficient conditions for a family of subspaces to be an absolutely representing family of subspaces and
study properties of absolutely representing families of subspaces in Banach spaces.

As an example, we study families of subspaces spanned by $e^{-\alpha t}$ in the space $C_\infty[0,\infty)$.

\textbf{Keywords:} Banach space, absolutely representing family of subspaces.
\end{abstract}

\section{Introduction}

Let $X$ be a Banach space over a field $\mathbb{K}$ of real or complex numbers, and $\Lambda$ be a nonempty set.
Suppose that $X_\lambda$ is a subspace of $X$ (by a subspace we mean a closed lineal), $\lambda\in\Lambda$.
We denote this family of subspaces by $S=(X;X_\lambda\mid\lambda\in\Lambda)$.
The family $S$ is called an \textsl{absolutely representing family of subspaces} (ARFS) in $X$
if for every $x\in X$ there exist $x_\lambda\in X_\lambda$, $\lambda\in\Lambda$, such that
$\sum_{\lambda\in\Lambda}\|x_\lambda\|<\infty$ and $x=\sum_{\lambda\in\Lambda}x_\lambda$.

If $\Lambda$ is countable, then the definition of an ARFS coincides with the definition of
an \textsl{absolutely representing system of subspaces} (ARSS) in $X$ (see, for example, \cite{KorRSS}).

The notion of an ARFS is closely related to the notion of an \textsl{absolutely representing family} (ARF) in $X$ (see, for example, \cite{Kor_ARF}).
Let us recall the corresponding definition.
Let $\Omega$ be a nonempty set.
Suppose that $v_\omega$ is a nonzero element of $X$, $\omega\in\Omega$.
We denote this family of elements by $V=(X;v_\omega\mid\omega\in\Omega)$.
The family $V$ is called an ARF in $X$ if for every $x\in X$ there exist $c_\omega\in\mathbb{K}$, $\omega\in\Omega$,
such that $\sum_{\omega\in\Omega}\|c_\omega v_\omega\|<\infty$ and $x=\sum_{\omega\in\Omega}c_\omega v_\omega$.
Let us show a connection between absolutely representing families of subspaces and absolutely representing systems (of elements).
For a subspace $Y\subset X$, define $\mathbf{B}_Y=\{y\in Y\mid \|y\|=1\}$.
Clearly, $S=(X;X_\lambda\mid\lambda\in\Lambda)$ is an ARFS in $X$ iff the family
\begin{equation*}
V(S)=(X;\bigcup_{\lambda\in\Lambda}\mathbf{B}_{X_\lambda})
\end{equation*}
is an ARF in $X$.
For a nonzero element $v\in X$, denote by $\langle v\rangle$ the one-dimensional subspace spanned by $v$.
Clearly, $V=(X;v_\omega\mid\omega\in\Omega)$ is an ARF in $X$ iff the family of one-dimensional subspaces
\begin{equation*}
S(V)=(X;\langle v_\omega\rangle\mid\omega\in\Omega)
\end{equation*}
is an ARFS in $X$.

One can give the definitions of an ARSS and an ARF for a wider class of spaces than the Banach spaces, e.g., for
the complete Hausdorff locally convex spaces.
ARSS and ARF in various classes of spaces were studied in numerous publications
and have many applications, e.g., in complex analysis
(the problem of representation of the functions analytic in a domain $D\subset\mathbb{C}$ by series of exponents)
(see, for instance, \cite{Kor_ARF},\cite{KorRSS},\cite{Korobeinik_09},\cite{Abanin_95},\cite{Abanin_00},\cite{Abanin_06},\cite{Mihaylov}).

ARF in Banach spaces were studied in~\cite[p. 209-210]{Banach}, \cite{Vershynyn_1}, \cite{Vershynin_2};
ARF in Hilbert spaces were studied in~\cite{Shraifel_1}, \cite{Shraifel_2}.

In this paper we study ARFS in Banach spaces.

This paper is organized as follows.

In Section~\ref{S:Aux_results}, we provide auxiliary notions and results which will be useful in the next sections.

In Section~\ref{S:construction}, we provide some natural examples of ARFS.
We also give a method of construction of ARFS.

In Section~\ref{S:criterion_stability}, we obtain a criterion for a family of subspaces to be an ARFS in $X$.
Using this criterion, we prove the stability of ARFS.
More precisely, if $S$ is an ARFS in $X$, and a family of subspaces $\widetilde{S}=(X,\widetilde{X}_\lambda\mid\lambda\in\Lambda)$
is such that for every $\lambda\in\Lambda$ the subspace $X_\lambda$ is sufficiently close to $\widetilde{X}_\lambda$,
then $\widetilde{S}$ is also an ARFS in $X$.
As a measure of closeness of a subspace $Y$ to a subspace $Z$ we consider the quantity
\begin{equation}\label{E:rho_0}
\rho_0(Y,Z)=\sup_{y\in\mathbf{B}_Y}d(y,Z),
\end{equation}
where $d(x,F)$ is the distance from $x\in X$ to a set $F\subset X$ (if $Y=0$, then we set $\rho_0(Y,Z)=0$).

In Section~\ref{S:necessary_condition}, we obtain a necessary condition for a family of subspaces to be an ARFS in $X$.

In Section~\ref{S:sufficient_conditions_appl}, we get sufficient conditions for a family of subspaces to be an ARFS in $X$.
Using this results, we obtain sufficient conditions for a subfamily of an ARFS in $X$ to be an ARFS in $X$.

In Section~\ref{S:exponents}, we study families of subspaces spanned by $e^{-\alpha t}$ in $C_\infty[0,\infty)$.
Let us formulate the main problem studied in this section.
Denote by $C_\infty[0,\infty)$ the set of all continuous functions $f:[0,\infty)\to\mathbb{K}$ such that $\lim_{t\to\infty}f(t)=0$.
Set $\|f\|=\sup_{t\in[0,\infty)}|f(t)|$, $f\in C_\infty[0,\infty)$.
Define $I_n=\{1,\ldots,n\}$, $n\in\mathbb{N}$.
We also set $I_\infty=\{1,2,\ldots\}$.
Let $\Lambda$ be a nonempty set.
Suppose that $n(\lambda)\in\mathbb{N}\cup\{\infty\}$, $\lambda\in\Lambda$.
Let $\alpha(\lambda,k)$, $k\in I_{n(\lambda)}$, be pairwise distinct positive numbers, $\lambda\in\Lambda$.
Define $X_\lambda$ to be the subspace of spanned by $e^{-\alpha(\lambda,k)t}$, $k\in I_{n(\lambda)}$.

In Section~\ref{S:exponents} we study the following question:
\begin{center}
when the system of subspaces $X_\lambda$, $\lambda\in\Lambda$, is an ARFS in $C_\infty[0,\infty)$?
\end{center}
The answer depends on the family of numbers $\beta(\lambda)=\sum_{k\in I_{n(\lambda)}}1/\alpha(\lambda,k)$, $\lambda\in\Lambda$
(see Section~\ref{SS:exponents_results}).

\section{Auxiliary notions and results.
Spaces $\ell_1(Z_\lambda\mid\lambda\in\Lambda)$ and $\ell_\infty(Z_\lambda\mid\lambda\in\Lambda)$}\label{S:Aux_results}

Let $Z_\lambda$, $\lambda\in\Lambda$, be a family of Banach spaces.

Define $\ell_1(Z_\lambda\mid\lambda\in\Lambda)$ to be the linear space of elements
$\xi=(z_\lambda\mid\lambda\in\Lambda)$ such that $\sum_{\lambda\in\Lambda}\|z_\lambda\|<\infty$, endowed
with the norm $\|\xi\|_1=\sum_{\lambda\in\Lambda}\|z_\lambda\|$. It is easy to check that
$\ell_1(Z_\lambda\mid\lambda\in\Lambda)$ is a Banach space.

Define $\ell_\infty(Z_\lambda\mid\lambda\in\Lambda)$ to be the linear space of elements
$\xi=(z_\lambda\mid\lambda\in\Lambda)$ such that $\sup_{\lambda\in\Lambda}\|z_\lambda\|<\infty$, endowed
with the norm $\|\xi\|_\infty=\sup_{\lambda\in\Lambda}\|z_\lambda\|$. It is easy to check that
$\ell_\infty(Z_\lambda\mid\lambda\in\Lambda)$ is a Banach space.

For a Banach space $Z$ denote by $Z^*$ the linear space of all continuous linear mappings
$\varphi:Z\to\mathbb{K}$, endowed with the norm $\|\varphi\|=\sup_{z\in\mathbf{B}_Z}|\varphi(z)|$.

In the sequel, we will use the following relation between the spaces
$\ell_1(Z_\lambda\mid\lambda\in\Lambda)$ and $\ell_\infty(Z_\lambda\mid\lambda\in\Lambda)$:
\begin{equation*}
(\ell_1(Z_\lambda\mid\lambda\in\Lambda))^*=\ell_\infty(Z_\lambda^*\mid\lambda\in\Lambda).
\end{equation*}
Note that
\begin{equation*}
\xi^*(\xi)=\sum_{\lambda\in\Lambda}z_\lambda^*(z_\lambda)
\end{equation*}
for $\xi^*=(z_\lambda^*\mid\lambda\in\Lambda)\in\ell_\infty(Z_\lambda^*\mid\lambda\in\Lambda)$ and
$\xi=(z_\lambda\mid\lambda\in\Lambda)\in\ell_1(Z_\lambda\mid\lambda\in\Lambda)$.

\section{Natural examples of ARFS and a construction of ARFS}\label{S:construction}

\subsection{Natural examples of ARFS}

\begin{example}
Let $Z_\lambda$, $\lambda\in\Lambda$, be a family of Banach spaces.
Set $X=\ell_1(Z_\lambda\mid\lambda\in\Lambda)$.
Let $\Gamma_\lambda:Z_\lambda\to X$ be the natural embedding of $Z_\lambda$ into $X$.
Define $X_\lambda=\Gamma_\lambda Z_\lambda$, $\lambda\in\Lambda$.
Clearly, the family $X_\lambda$, $\lambda\in\Lambda$, is an ARFS in $X$.
\end{example}

\begin{example}
Let $(T,\mathcal{F},\mu)$ be a measure space.
Set $X=L_1(T,\mathcal{F},\mu)$, and let $\|\cdot\|=\|\cdot\|_1$, where $\|x(t)\|_1=\int_T|x(t)|\,dt$.
For $A\in\mathcal{F}$, define $X_A$ to be the set of $x(t)\in X$ such that $x(t)=0$ a.e. on $T\setminus A$.
Clearly, $X_A$ is a subspace of $X$.

Suppose sets $T_k\in\mathcal{F}$, $k\geqslant 1$, satisfy $\bigcup_{k=1}^\infty T_k=T$.
Then the family of subspaces $X_k=X_{T_k}$, $k\geqslant 1$, is an ARFS in $X$.

To prove this, define the sets $\widetilde{T}_k$, $k\geqslant 1$, by
$\widetilde{T}_1=T_1$, $\widetilde{T}_k=T_k\setminus(T_1\cup\ldots\cup T_{k-1})$ for $k\geqslant 2$.
Then $\widetilde{T}_k\subset T_k$, $\bigcup_{k=1}^\infty\widetilde{T}_k=T$, and $\widetilde{T}_i\cap\widetilde{T}_j=\varnothing$ for $i\neq j$.
Take any $x\in X$.
Define $x_k=x\mathbb{I}_{\widetilde{T}_k}\in X_k$, $k\geqslant 1$, where $\mathbb{I}_A$ is the indicator function of a set $A\in\mathcal{F}$.
It is easily seen that $x=\sum_{k=1}^\infty x_k$ (the series converges in $\|\cdot\|$).
Since $\|x_k\|=\int_{\widetilde{T}_k}|x_k|\,dt$, we conclude that $\sum_{k=1}^\infty\|x_k\|=\|x\|$.
Hence, $X_k$, $k\geqslant 1$, is an ARFS in $X$.
\end{example}

\begin{example}\label{EX:dense}
Let $X$ be a Banach space, and $X_\lambda$, $\lambda\in\Lambda$, be a family of subspaces of $X$.
Suppose $\bigcup_{\lambda\in\Lambda}X_\lambda$ is dense in $X$.

We claim that $X_\lambda$, $\lambda\in\Lambda$, is an ARFS in $X$.
Moreover, for any $x\in X$ and $\varepsilon>0$ there exists a family $x_\lambda\in X_\lambda$, $\lambda\in\Lambda$, such that
$\sum_{\lambda\in\Lambda}\|x_\lambda\|\leqslant\|x\|+\varepsilon$ and $x=\sum_{\lambda\in\Lambda}x_\lambda$.

Let us prove this.
Fix $x\in X$ and $\varepsilon>0$.
There exist $\lambda(1)\in\Lambda$ and $z_1\in X_{\lambda(1)}$ such that $\|x-z_1\|<\varepsilon/2^2$.
Suppose $z_1,\ldots,z_m$ have already been defined.
There exist $\lambda(m+1)\in\Lambda$ and $z_{m+1}\in X_{\lambda(m+1)}$ such that $\|(x-\sum_{k=1}^{m}z_k)-z_{m+1}\|<\varepsilon/2^{(m+2)}$.

We thus obtain the sequences $\lambda(m)\in\Lambda$, $m\geqslant 1$, and $z_m\in X_{\lambda(m)}$, $m\geqslant 1$.
By the construction of this sequences, we have $\|x-z_1-\ldots-z_m\|<\varepsilon/2^{(m+1)}$ for $m\geqslant 1$.
Hence, $x=\sum_{m=1}^{\infty}z_m$.
Moreover, $\|z_1\|<\|x\|+\varepsilon/2^2$ and $\|z_m\|<\varepsilon/2^{(m+1)}+\varepsilon/2^m$ for $m\geqslant 2$.
It follows that $\sum_{m=1}^{\infty}\|z_m\|<\|x\|+\varepsilon$.

Define $x_\lambda=\sum_{m:\lambda(m)=\lambda}z_m$ (the empty sum is defined to be $0$), $\lambda\in\Lambda$.
Obviously, $x_\lambda\in X_\lambda$, $\lambda\in\Lambda$.
We have $x=\sum_{\lambda\in\Lambda}x_\lambda$ and $\sum_{\lambda\in\Lambda}\|x_\lambda\|\leqslant\sum_{m=1}^{\infty}\|z_m\|<\|x\|+\varepsilon$.
\end{example}

\subsection{A construction of ARFS}

For a subset $M$ of a Banach space $Y$, define $\overline{M}$ to be the closure of $M$.

The following obvious proposition shows that if a family of subspaces $S$ is an ARFS in $X$, then
every surjective operator $A:X\to Y$ generates in a natural way an ARFS in $Y$.

\begin{proposition}
Let $X,Y$ be Banach spaces, $A:X\to Y$ a continuous linear operator with $\mathrm{Im}(A)=Y$.
If $S=(X;X_\lambda\mid\lambda\in\Lambda)$ is an ARFS in $X$, then the family of subspaces
$A(S)=(Y;\overline{A(X_\lambda)}\mid\lambda\in\Lambda)$ is an ARFS in $Y$.
\end{proposition}

\section{Criterion for $S$ to be an ARFS in $X$ and its applications}\label{S:criterion_stability}

Let $X$ be a Banach space, and $S=(X;X_\lambda\mid\lambda\in\Lambda)$ be a family of subspaces of $X$.

\subsection{Criterion for $S$ to be an ARFS in $X$}\label{SS:criterion}

The following criterion for $S$ to be an ARFS in $X$ generalizes the well-known criteria
for a family of vectors to be an ARF in $X$ (see, e.g., \cite[Theorem 1]{Korobeinik_09}, \cite[Theorem 2.1, (iv)]{Vershynin_2}),
and for a system of subspaces to be an ARSS in $X$ (\cite[Theorem 1]{Abanin_00}).

For a mapping $f:X\to\mathbb{K}$ and a set $Y\subset X$, denote by $f\upharpoonright_Y$ the restriction of $f$ to $Y$.

\begin{theorem}\label{T:criterion_ARFS}
$S$ is an ARFS in $X$ if and only if there exists an $\varepsilon>0$ such that
\begin{equation}\label{E:ARFS}
\sup_{\lambda\in\Lambda}\|\varphi\upharpoonright_{X_\lambda}\|\geqslant\varepsilon\|\varphi\|
\end{equation}
for any $\varphi\in X^*$.
\end{theorem}
\begin{proof}
Consider the continuous linear operator $A:\ell_1(X_\lambda\mid\lambda\in\Lambda)\to X$ defined by $A(x_\lambda\mid\lambda\in\Lambda)=\sum_{\lambda\in\Lambda}x_\lambda$.
Then $A^*:X^*\to\ell_\infty(X_\lambda^*\mid\lambda\in\Lambda)$.
It is easy to check that $A^*\varphi=(\varphi\upharpoonright_{X_\lambda}\mid\lambda\in\Lambda)$.
$S$ is an ARFS in $X$ $\Leftrightarrow$
$\mathrm{Im}(A)=X$ $\Leftrightarrow$
$A^*$ is an isomorphic embedding, that is, there exists an $\varepsilon>0$ such that
$\|A^*\varphi\|_\infty\geqslant\varepsilon\|\varphi\|$, $\varphi\in X^*$.
This completes the proof.
\end{proof}

Let us formulate a criterion for a family of subspaces to be an ARFS in geometric terms.
Recall that as a measure of closeness of a subspace $Y$ to a subspace $Z$ we consider the quantity $\rho_0(Y,Z)$ defined by~\eqref{E:rho_0}.

\begin{theorem}\label{T:criterion_ARFS_geom}
$S$ is an ARFS in $X$ if and only if there exists an $\varepsilon>0$ such that
\begin{equation}\label{E:ARFS_geom}
\sup_{\lambda\in\Lambda}\rho_0(X_\lambda,Y)\geqslant\varepsilon
\end{equation}
for any subspace $Y\subset X$, $Y\neq X$.
\end{theorem}
\begin{proof}
We make the following key observation.
Suppose that $\varphi\in X^*$, $\|\varphi\|=1$.
For any subspace $Z\subset X$ we have
\begin{equation*}
\rho_0(Z,\ker(\varphi))=\sup_{x\in\mathbf{B}_Z}d(x,\ker(\varphi))=\sup_{x\in\mathbf{B}_Z}|\varphi(x)|=\|\varphi\upharpoonright_Z\|.
\end{equation*}

$(\Rightarrow)$ 
Consider a subspace $Y\subset X$, $Y\neq X$.
Clearly, there exists $\varphi\in X^*$, $\|\varphi\|=1$, such that $Y\subset\ker(\varphi)$.
For any $\lambda\in\Lambda$, we have $\rho_0(X_\lambda,Y)\geqslant\rho_0(X_\lambda,\ker(\varphi))=\|\varphi\upharpoonright_{X_\lambda}\|$.
Using Theorem~\ref{T:criterion_ARFS}, we get the required assertion.

$(\Leftarrow)$
Consider any $\varphi\in X^*$, $\|\varphi\|=1$.
Set $Y=\ker(\varphi)$. Using~\eqref{E:ARFS_geom}, we get $\sup_{\lambda\in\Lambda}\|\varphi\upharpoonright_{X_\lambda}\|\geqslant\varepsilon$.
From Theorem~\ref{T:criterion_ARFS} it follows that $S$ is an ARFS in $X$.
\end{proof}

\subsection{An application of Theorem~\ref{T:criterion_ARFS}. The stability of ARFS in $X$}\label{SS:stability}

\begin{theorem}\label{T:stability}
If $S=(X;X_\lambda\mid\lambda\in\Lambda)$ is an ARFS in $X$,
then there exists $R>0$ with the following property:
if a family of subspaces $\widetilde{S}=(X;\widetilde{X}_\lambda\mid\lambda\in\Lambda)$ satisfies
$\sup_{\lambda\in\Lambda}\rho_{0}(X_\lambda,\widetilde{X}_\lambda)<R$, then $\widetilde{S}$ is an ARFS in $X$.
\end{theorem}

Theorem~\ref{T:stability} is a consequence of the following lemma and Theorem~\ref{T:criterion_ARFS}.

\begin{lemma}\label{L:stability}
Let $S=(X;X_\lambda\mid\lambda\in\Lambda)$ be an ARFS in $X$. Suppose $\varepsilon>0$ is such that
the inequality~\eqref{E:ARFS} holds for any $\varphi\in X^*$.
If $r=\sup_{\lambda\in\Lambda}\rho_0(X_\lambda,\widetilde{X}_\lambda)<\varepsilon$, then
\begin{equation*}
\sup_{\lambda\in\Lambda}\|\varphi\upharpoonright_{\widetilde{X}_\lambda}\|\geqslant\frac{\varepsilon-r}{1+r}\|\varphi\|,\quad \varphi\in X^*.
\end{equation*}
\end{lemma}
\begin{proof}
Consider arbitrary $\varepsilon_1,r_1$ such that $r<r_1<\varepsilon_1<\varepsilon$.
Let $\varphi\in X^*$, $\|\varphi\|=1$.
There exist $\lambda\in\Lambda$ and $x\in X_\lambda$, $\|x\|=1$, such that $|\varphi(x)|>\varepsilon_1$.
Since $\rho_0(X_\lambda,\widetilde{X}_\lambda)\leqslant r$, we have $d(x,\widetilde{X}_\lambda)\leqslant r$.
Hence, there exists $\widetilde{x}\in \widetilde{X}_\lambda$ such that $\|x-\widetilde{x}\|<r_1$.
Then
\begin{equation*}
|\varphi(\widetilde{x})|\geqslant|\varphi(x)|-|\varphi(x-\widetilde{x})|\geqslant\varepsilon_1-r_1.
\end{equation*}
Since $\|\widetilde{x}\|\leqslant 1+r_1$, we have
\begin{equation*}
\|\varphi\upharpoonright_{\widetilde{X}_\lambda}\|\geqslant\frac{\varepsilon_1-r_1}{1+r_1}.
\end{equation*}
Since $\varepsilon_1,r_1$ were arbitrary, we get the needed assertion.
\end{proof}

\section{A necessary condition for $S$ to be an ARFS in $X$}\label{S:necessary_condition}

Let $X$ be a Banach space, and $S=(X;X_\lambda\mid\lambda\in\Lambda)$ be a family of subspaces of $X$.

If $S$ is an ARFS in $X$, then a more stronger result than the inequality~\eqref{E:ARFS} holds.
To formulate this result we need to introduce a few auxiliary notions.
Let $f:X\to\mathbb{R}$ be a mapping.
Consider the following three conditions:\\
\textbf{(ZZ)} $f(0)=0$;\\
\textbf{(CZ)} $f$ is continuous at $0$;\\
\textbf{(SA)} (the subadditivity of $f$) $f(x+y)\leqslant f(x)+f(y)$ for any $x,y\in X$.\\
Define
\begin{equation*}
\|f\|=\sup_{x\neq 0}\frac{f(x)}{\|x\|}.
\end{equation*}
Note that, if $f$ satisfies (ZZ) and (SA), then $\|f\|\geqslant 0$ (we have $f(x)+f(-x)\geqslant 0$, hence, $\max\{f(x),f(-x)\}\geqslant 0$).
Let us provide some examples of mappings $f$ which satisfy (ZZ), (CZ), and (SA).

\begin{example}
If $f$ is a seminorm on $X$, and $f$ is continuous with respect to $\|\cdot\|$, then
$f$ satisfies (ZZ), (CZ), and (SA).
\end{example}

\begin{example}
Let $(T,\mathcal{F},\mu)$ be a measure space with $\mu(X)\in(0,\infty)$.
Let $p\in[1,\infty)$.
Set $X=L_p(T,\mathcal{F},\mu)$, $\|\cdot\|=\|\cdot\|_p$, where $\|x(t)\|_p=(\int_T|x(t)|^p\,dt)^{1/p}$.
If $r\in[1,p]$, then we can define $f(x)=\|x\|_r$, $x\in X$.
Clearly, $f$ satisfies (ZZ), (CZ), (SA), and $\|f\|=(\mu(X))^{1/r-1/p}$.
\end{example}

\begin{example}
Let a set $Y\subset X$ be such that $0\in Y$ and $y_1+y_2\in Y$ for any $y_1,y_2\in Y$.
Then the mapping $f(x)=d(x,Y)$ satisfies (ZZ), (CZ), (SA).
\end{example}

\begin{theorem}\label{T:necessary_ARFS}
Suppose $S$ is an ARFS in $X$.
Then there exists an $\varepsilon>0$ which possesses the following property:
if a mapping $f:X\to\mathbb{R}$ satisfies (ZZ), (CZ), (SA), then
\begin{equation*}
\sup_{\lambda\in\Lambda}\|f\upharpoonright_{X_\lambda}\|\geqslant\varepsilon\|f\|
\end{equation*}
(if $X_\lambda=0$, then we set $\|f\upharpoonright_{X_\lambda}\|=0$).
\end{theorem}
\begin{proof}
We consider the case where $\Lambda$ is infinite, a proof for the case $|\Lambda|<\infty$ is similar.
Define the operator $A:\ell_1(X_\lambda\mid\lambda\in\Lambda)\to X$ by $A(x_\lambda\mid\lambda\in\Lambda)=\sum_{\lambda}x_\lambda$.
Since $S$ is an ARFS in $X$, we conclude that $\mathrm{Im}(A)=X$.
By the open mapping theorem, there exists $M>0$ such that for any
$x\in X$ there exist $x_\lambda\in X_\lambda$, $\lambda\in\Lambda$, such that
\begin{equation}\label{E:open_mapping}
\sum_{\lambda\in\Lambda}x_\lambda=x \quad\text{and}\quad \sum_{\lambda\in\Lambda}\|x_\lambda\|\leqslant M\|x\|.
\end{equation}
Consider any $f:X\to\mathbb{R}$ satisfying (ZZ), (CZ), (SA).
Fix any $x\in X$.
Let $x_\lambda\in X_\lambda$, $\lambda\in\Lambda$, satisfy~\eqref{E:open_mapping}.
There exists a countable set $\{\lambda_j\mid j\geqslant 1\}\subset\Lambda$ such that $x_\lambda=0$ for any $\lambda\notin\{\lambda_j\mid j\geqslant 1\}$.
For any $n\in\mathbb{N}$ we have
\begin{align*}
&f(x)\leqslant\sum_{j=1}^n f(x_{\lambda_j})+f(x-\sum_{j=1}^n x_{\lambda_j})\leqslant
\sum_{j=1}^n\|f\upharpoonright_{X_{\lambda_j}}\|\|x_{\lambda_j}\|+f(x-\sum_{j=1}^n x_{\lambda_j})\leqslant\\
&\leqslant\sup_{\lambda\in\Lambda}\|f\upharpoonright_{X_\lambda}\|\sum_{j=1}^n\|x_{\lambda_j}\|+f(x-\sum_{j=1}^n x_{\lambda_j})
\leqslant M\|x\|\sup_{\lambda\in\Lambda}\|f\upharpoonright_{X_\lambda}\|+f(x-\sum_{j=1}^n x_{\lambda_j}).
\end{align*}
Letting $n\to\infty$ and using (CZ), we get
\begin{equation*}
f(x)\leqslant M\|x\|\sup_{\lambda\in\Lambda}\|f\upharpoonright_{X_\lambda}\|\Rightarrow
\sup_{\lambda\in\Lambda}\|f\upharpoonright_{X_\lambda}\|\geqslant\frac{1}{M}\frac{f(x)}{\|x\|}.
\end{equation*}
This proves the required assertion with $\varepsilon=1/M$.
\end{proof}

\section{Sufficient conditions for $S$ to be an ARFS in $X$ and their applications}\label{S:sufficient_conditions_appl}

Let $X$ be a Banach space, and $S=(X;X_\lambda\mid\lambda\in\Lambda)$ be a family of subspaces of $X$.

\subsection{Sufficient conditions for $S$ to be an ARFS in $X$}\label{SS:sufficient_conditions}

By Theorem~\ref{T:criterion_ARFS}, $S$ is an ARFS in $X$ if and only if there exists an $\varepsilon>0$ such that
\begin{equation}\label{E:ineq_ARFS_sufficient}
\sup_{\lambda\in\Lambda}\|\varphi\upharpoonright_{X_\lambda}\|\geqslant\varepsilon\|\varphi\|
\end{equation}
for any $\varphi\in X^*$.

Let $Y$ be a subspace of $X$.
Denote by $Y^{\bot}$ the set of all $\varphi\in X^*$ such that $\varphi\upharpoonright_{Y}=0$.
Under some additional conditions we show that if the inequality~\eqref{E:ineq_ARFS_sufficient} holds for
$\varphi\in Y^\bot$, then $S$ is an ARFS in $X$.

Define $\mathcal{L}(X_\lambda\mid\lambda\in\Lambda)$ to be the linear span of $X_\lambda$, $\lambda\in\Lambda$.
Recall that as a measure of closeness of a subspace $Y$ to a subspace $Z$ we consider the quantity $\rho_0(Y,Z)$ defined by~\eqref{E:rho_0}.

\begin{theorem}\label{T:criterion_restriction_phi}
Let $Y$ be a subspace of $X$.
Suppose that for any $\delta>0$ there exists a subspace $L\subset\mathcal{L}(X_\lambda\mid\lambda\in\Lambda)$ with $\rho_0(Y,L)<\delta$.
If there exists $\varepsilon>0$ such that
\begin{equation*}
\sup_{\lambda\in\Lambda}\|\varphi\upharpoonright_{X_\lambda}\|\geqslant\varepsilon\|\varphi\|,\quad\varphi\in Y^{\bot},
\end{equation*}
then $S$ is an ARFS in $X$.
\end{theorem}
\begin{proof}
Suppose that $S$ is not an ARFS in $X$.

Take an arbitrary $\delta\in(0,1/3)$.
There exists a subspace $L\subset\mathcal{L}(X_\lambda\mid\lambda\in\Lambda)$ with $\rho_0(Y,L)<\delta$.
Clearly, the family $L$, $X_\lambda$, $\lambda\in\Lambda$, is not an ARFS in $X$.
From Theorem~\ref{T:criterion_ARFS} it follows that there exists $\varphi\in X^{*}$, $\|\varphi\|=1$, such that
$\|\varphi\upharpoonright_{L}\|<\delta$ and
$\|\varphi\upharpoonright_{X_\lambda}\|<\delta$ for $\lambda\in\Lambda$.

Let us estimate from above $\|\varphi\upharpoonright_Y\|$.
Take any $y\in Y$, $y\neq 0$.
From $\rho_0(Y,L)<\delta$ it follows that $d(y,L)<\delta\|y\|$.
Hence, there exists $l\in L$ such that $\|y-l\|\leqslant\delta\|y\|$.
Then $\|l\|\leqslant (1+\delta)\|y\|\leqslant 2\|y\|$.
We have
\begin{equation*}
|\varphi(y)|\leqslant|\varphi(y-l)|+|\varphi(l)|\leqslant\|y-l\|+\delta\|l\|\leqslant \delta\|y\|+2\delta\|y\|=3\delta\|y\|.
\end{equation*}
Hence, $\|\varphi\upharpoonright_Y\|\leqslant 3\delta$.

By the Hahn-Banach theorem, there exists a $\psi\in X^*$ such that $\psi\upharpoonright_Y=\varphi\upharpoonright_Y$ and
$\|\psi\|\leqslant\|\varphi\upharpoonright_Y\|\leqslant 3\delta$.
Set $\eta=\varphi-\psi$.
Then $\eta\in Y^\bot$, $\|\eta\|\geqslant\|\varphi\|-\|\psi\|\geqslant 1-3\delta$, and
\begin{equation*}
\|\eta\upharpoonright_{X_\lambda}\|\leqslant\|\varphi\upharpoonright_{X_\lambda}\|+\|\psi\upharpoonright_{X_\lambda}\|\leqslant
\delta+3\delta=4\delta
\end{equation*}
for any $\lambda\in\Lambda$.
Define $\widetilde{\eta}=\eta/\|\eta\|$.
Then $\widetilde{\eta}\in Y^{\bot}$, $\|\widetilde{\eta}\|=1$, and
$\|\widetilde{\eta}\upharpoonright_{X_\lambda}\|\leqslant 4\delta/(1-3\delta)$ for any $\lambda\in\Lambda$.
Since $\delta$ was arbitrary, we get a contradiction.
The proof is complete.
\end{proof}

We note the following corollary of Theorem~\ref{T:criterion_restriction_phi}, which generalizes Theorem~3 of~\cite{Shraifel_2}.

\begin{theorem}\label{T:criterion_cofinite_dim}
Let $Y$ be a finite dimensional subspace of $X$.
Suppose that $\mathcal{L}(X_\lambda\mid\lambda\in\Lambda)$ is dense in $X$.
If there exists an $\varepsilon>0$ such that
\begin{equation}\label{E:ineq_cofinite_dim}
\sup_{\lambda\in\Lambda}\|\varphi\upharpoonright_{X_\lambda}\|\geqslant\varepsilon\|\varphi\|,\quad\varphi\in Y^{\bot},
\end{equation}
then $S$ is an ARFS in $X$.
\end{theorem}
\begin{proof}
Let us prove that for any $\delta>0$ there exists a subspace $L\subset\mathcal{L}(X_\lambda\mid\lambda\in\Lambda)$ with $\rho_0(Y,L)<\delta$.

Let $y_{1},\ldots,y_{m}$ be a basis of $Y$ with $\|y_j\|=1$ for $j=1,\ldots,m$.
There exists $c>0$ such that $\|\sum_{k=1}^{m}t_{k}y_{k}\|\geqslant c\sum_{k=1}^{m}|t_{k}|$ for any $t_1,\ldots,t_m\in\mathbb{K}$.
Fix any $\delta_1>0$.
Since $\mathcal{L}(X_\lambda\mid\lambda\in\Lambda)$ is dense in $X$, there exist
$l_{1},\ldots,l_{m}\in\mathcal{L}(X_\lambda\mid\lambda\in\Lambda)$ such that $\|y_{k}-l_{k}\|<\delta_1$, $\|l_{k}\|=1$ for $k=1,\ldots,m$.
Define $L$ to be the subspace spanned by $l_1,\ldots,l_m$.
Let $y\in Y$.
Then $y=\sum_{k=1}^m t_k y_k$ for some $t_1,\ldots,t_m\in\mathbb{K}$.
Define $l=\sum_{k=1}^m t_k l_k$.
We have
\begin{equation*}
\|y-l\|=\|\sum_{k=1}^m t_k(y_k-l_k)\|\leqslant\delta_1\sum_{k=1}^m|t_k|\leqslant \delta_1 c^{-1}\|y\|.
\end{equation*}
Hence, $d(y,L)\leqslant \delta_1 c^{-1}\|y\|$, whence $\rho_0(Y,L)\leqslant \delta_1 c^{-1}$.

By Theorem~\ref{T:criterion_restriction_phi}, $S$ is an ARFS in $X$.
\end{proof}

\begin{remark}
Theorem~\ref{T:criterion_cofinite_dim} is not valid without the assumption that $Y$ is finite dimensional.
To see this, consider the following example.
Let $X$ be a Hilbert space. Then we can identify $X^*$ with $X$.
Let $Y$ be a subspace of $X$, $\dim Y=\infty$.
Let $e_i$, $i\in I$, be an orthonormal basis of $Y$.
We assume that $0\notin I$.
Define $X_0=Y^{\bot}$, $X_i=\langle e_i\rangle$ (the one-dimensional subspace spanned by $e_i$), $i\in I$.
Clearly, $\mathcal{L}(X_0,X_i\mid i\in I)$ is dense in $X$, and the inequality~\eqref{E:ineq_cofinite_dim} holds (we can set $\varepsilon=1$).
But $X_0$, $X_i$, $i\in I$, is not an ARFS in $X$.
\end{remark}

\begin{remark}
Clearly, Theorem~\ref{T:criterion_cofinite_dim} is not valid without the assumption that $\mathcal{L}(X_\lambda\mid\lambda\in\Lambda)$
is dense in $X$.
\end{remark}

\subsection{An application of Theorems~\ref{T:criterion_restriction_phi}, \ref{T:criterion_cofinite_dim}.
On subfamilies of ARFS in $X$}\label{SS:delete_ARFS}

Suppose $S=(X;X_\lambda\mid\lambda\in\Lambda)$ is an ARFS in $X$.
Let $\Lambda'$ be a subset of $\Lambda$.
We delete the subspaces $X_\lambda$, $\lambda\in\Lambda'$, from $S$.
Hence, we get the family of subspaces $X_\lambda$, $\lambda\in\Lambda\setminus\Lambda'$.
We will give sufficient conditions for this family of subspaces to be an ARFS in $X$.

Denote by $\langle X_\lambda\mid\lambda\in\Lambda'\rangle$ the subspace spanned by $X_\lambda$, $\lambda\in\Lambda'$.

\begin{theorem}\label{T:ARFS_delete}
Let $X_\lambda$, $\lambda\in\lambda$, be an ARFS in $X$, and $\Lambda'$ be a subset of $\Lambda$.
Suppose that for any $\delta>0$ there is a subspace $L\subset\mathcal{L}(X_\lambda\mid\lambda\in\Lambda\setminus\Lambda')$
such that $\rho_0(\langle X_\lambda\mid\lambda\in\Lambda'\rangle,L)<\delta$.

Then $X_\lambda$, $\lambda\in\Lambda\setminus\Lambda'$, is an ARFS in $X$.
\end{theorem}
\begin{proof}
To prove the required assertion, we will use Theorem~\ref{T:criterion_restriction_phi}.

Set $Y=\langle X_\lambda\mid\lambda\in\Lambda'\rangle$.
Then for any $\delta>0$ there is a subspace $L\subset\mathcal{L}(X_\lambda\mid\lambda\in\Lambda\setminus\Lambda')$ such that $\rho_0(Y,L)<\delta$.

By Theorem~\ref{T:criterion_ARFS}, there is an $\varepsilon>0$ such that~\eqref{E:ineq_ARFS_sufficient} holds for any $\varphi\in X^*$.
Consider any $\varphi\in Y^\bot$.
Then $\varphi\upharpoonright_{X_\lambda}=0$, $\lambda\in\Lambda'$.
Using~\eqref{E:ineq_ARFS_sufficient}, we get
\begin{equation*}
\sup_{\lambda\in\Lambda\setminus\Lambda'}\|\varphi\upharpoonright_{X_\lambda}\|\geqslant\varepsilon\|\varphi\|.
\end{equation*}
From Theorem~\ref{T:criterion_restriction_phi} it follows that $X_\lambda$, $\lambda\in\Lambda\setminus\Lambda'$, is an ARFS in $X$.
\end{proof}

\begin{theorem}\label{T:ARFS_delete_finite}
Let $X_\lambda$, $\lambda\in\Lambda$, be an ARFS in $X$, and $\Lambda'$ be a subset of $\Lambda$.
Suppose that
\begin{enumerate}
\item
$\Lambda'$ is finite;
\item
$X_\lambda$ is finite dimensional for $\lambda\in\Lambda'$.
\end{enumerate}
Then the family $X_\lambda$, $\lambda\in\Lambda\setminus\Lambda'$, is an ARFS in $X$ if and only if
$\mathcal{L}(X_\lambda\mid\lambda\in\Lambda\setminus\Lambda')$ is dense in $X$.
\end{theorem}
\begin{proof}
Clearly, if $X_\lambda$, $\lambda\in\Lambda\setminus\Lambda'$, is an ARFS in $X$, then
$\mathcal{L}(X_\lambda\mid\lambda\in\Lambda\setminus\Lambda')$ is dense in $X$.

Now suppose that $\mathcal{L}(X_\lambda\mid\lambda\in\Lambda\setminus\Lambda')$ is dense in $X$.
To prove that $X_\lambda$, $\lambda\in\Lambda\setminus\Lambda'$, is an ARFS in $X$ we will use Theorem~\ref{T:criterion_cofinite_dim}.

Set $Y=\langle X_\lambda\mid\lambda\in\Lambda'\rangle$. 
Then $Y$ is a finite dimensional subspace of $X$.

By Theorem~\ref{T:criterion_ARFS}, there is an $\varepsilon>0$ such that~\eqref{E:ineq_ARFS_sufficient} holds for any $\varphi\in X^*$.
Consider any $\varphi\in Y^\bot$.
Then $\varphi\upharpoonright_{X_\lambda}=0$, $\lambda\in\Lambda'$.
Using~\eqref{E:ineq_ARFS_sufficient}, we get
\begin{equation*}
\sup_{\lambda\in\Lambda\setminus\Lambda'}\|\varphi\upharpoonright_{X_\lambda}\|\geqslant\varepsilon\|\varphi\|.
\end{equation*}
From Theorem~\ref{T:criterion_cofinite_dim} it follows that $X_\lambda$, $\lambda\in\Lambda\setminus\Lambda'$, is an ARFS in $X$.
\end{proof}

\section{Families of subspaces spanned by $e^{-\alpha t}$ in $C_\infty[0,\infty)$}\label{S:exponents}

Let us denote by $C_\infty[0,\infty)$ the set of all continuous functions $f:[0,\infty)\to\mathbb{K}$
such that $\lim_{t\to\infty}f(t)=0$.
Set $\|f\|=\sup_{t\in[0,\infty)}|f(t)|$, $f\in C_\infty[0,\infty)$.
Clearly, $C_\infty[0,\infty)$ is a Banach space with respect to the norm $\|\cdot\|$.

\subsection{}\label{SS:exponents_ARF}
The system of elements $e^{-\lambda t}$, $\lambda>0$, is not an absolutely representing system in $C_\infty[0,\infty)$.

Indeed, suppose that
\begin{equation}\label{E:exponents}
f(t)=\sum_{j=1}^\infty a_j e^{-\lambda_j t},\quad\sum_{j=1}^\infty\|a_j e^{-\lambda_j t}\|=\sum_{j=1}^\infty|a_j|<\infty
\end{equation}
(the series converges in $\|\cdot\|$).
Define $F(z)=\sum_{j=1}^\infty a_j e^{-\lambda_j z}$, $\mathrm{Re}(z)\geqslant 0$.
Clearly, $F$ is analytic in $\mathrm{Re}(z)>0$ and continuous in $\mathrm{Re}(z)\geqslant 0$.
Moreover, $F(t)=f(t)$ for $t\in[0,\infty)$.

Conclusion: if $f\in C_\infty[0,\infty)$ can be represented in the form~\eqref{E:exponents},
then there exists a function $F(z)$, $\mathrm{Re}(z)\geqslant 0$, such that
\begin{enumerate}
\item
$F$ is an extension of $f$, that is, $F(t)=f(t)$, $t\in[0,\infty)$;
\item
$F$ is analytic in $\mathrm{Re}(z)>0$ and continuous in $\mathrm{Re}(z)\geqslant 0$.
\end{enumerate}

But not every function $f\in C_\infty[0,\infty)$ possesses this property.
Hence, the system $e^{-\lambda t}$, $\lambda>0$, is not an absolutely representing system in $C_\infty[0,\infty)$.

\subsection{Formulation of the problem and main results}\label{SS:exponents_results}

It is natural to study families of subspaces spanned by $e^{-\alpha t}$ in $C_\infty[0,\infty)$.

Let us introduce some notation.
Define $I_n=\{1,\ldots,n\}$, $n\in\mathbb{N}$.
We also set $I_\infty=\{1,2,\ldots\}$.
Let $n\in\mathbb{N}\cup\{\infty\}$, and let $\alpha_k$, $k\in I_n$, be a sequence of pairwise distinct positive numbers.
Define $Exp(\alpha_k\mid k\in I_n)$ to be the subspace of $C_\infty[0,\infty)$ spanned by $e^{-\alpha_k t}$, $k\in I_n$.

Let $\Lambda$ be a nonempty set.
Suppose that $n(\lambda)\in\mathbb{N}\cup\{\infty\}$, $\lambda\in\Lambda$.
Let $\alpha(\lambda,k)$, $k\in I_{n(\lambda)}$, be pairwise distinct positive numbers, $\lambda\in\Lambda$.
In what follows we assume that
\begin{enumerate}
\item
if $n(\lambda)\in\mathbb{N}$, then $\alpha(\lambda,1)<\ldots<\alpha(\lambda,n(\lambda))$;
\item
if $n(\lambda)=\infty$, then $\alpha(\lambda,1)<\alpha(\lambda,2)<\ldots$.
\end{enumerate}

In this section we study the following question:
\begin{center}
\emph{when the system of subspaces $X_\lambda=Exp(\alpha(\lambda,k)\mid k\in I_{n(\lambda)})$, $\lambda\in\Lambda$, is an ARFS in $C_\infty[0,\infty)$?}
\end{center}

To formulate our first result, we need a few auxiliary definitions.
Define
\begin{equation*}
\beta(\lambda)=\sum_{k\in I_{n(\lambda)}}\frac{1}{\alpha(\lambda,k)}\in(0,\infty],\quad\lambda\in\Lambda.
\end{equation*}

\begin{remark}
The subspace $X_\lambda$ is equal to $C_\infty[0,\infty)$ iff $\beta(\lambda)=\infty$.

Indeed, denote by $C_0[0,1]$ the set of all continuous functions $g:[0,1]\to\mathbb{K}$ such that $g(0)=0$.
Set $\|g\|=\max_{x\in[0,1]}|g(x)|$, $g\in C_0[0,1]$.
Suppose that $0<\alpha_1<\alpha_2<\ldots$.
From Muntz's theorem (see, e.g., \cite[Section 4.2]{Borwein_Erdeliy}) it follows that the linear span of
$x^{\alpha_k}$, $k\geqslant 1$, is dense in $C_0[0,1]$ iff $\sum_{k=1}^{\infty}1/\alpha_k=\infty$.
Using the substitution $x=e^{-t}$, $t\in[0,\infty)$, we get the following result:
the linear span of $e^{-\alpha_k t}$, $k\geqslant 1$, is dense in $C_\infty[0,\infty)$ iff $\sum_{k=1}^\infty 1/\alpha_k=\infty$.
Hence, $X_\lambda=C_\infty[0,\infty)$ iff $\beta(\lambda)=\infty$.
\end{remark}

In what follows we assume that $\beta(\lambda)<\infty$, $\lambda\in\Lambda$.

The family of numbers $\beta(\lambda)$, $\lambda\in\Lambda$, is said to be bounded if there is a $C$ such that $\beta(\lambda)\leqslant C$, $\lambda\in\Lambda$.
Otherwise the family $\beta(\lambda)$, $\lambda\in\Lambda$, is said to be unbounded.

\begin{theorem}\label{T:exponents_ARFS}
Suppose that $\inf_{\lambda\in\Lambda}\alpha(\lambda,1)>0$.
If the family $\beta(\lambda)$, $\lambda\in\Lambda$, is unbounded, then $X_\lambda$, $\lambda\in\Lambda$, is an ARFS in $C_\infty[0,\infty)$.
\end{theorem}

\begin{remark}
Theorem~\ref{T:exponents_ARFS} is not valid without the assumption that $\inf_{\lambda\in\Lambda}\alpha(\lambda,1)>0$.
To see this, we consider the following example.
Set $\Lambda=\mathbb{N}$ and let $n(\lambda)=1$, $\lambda\in\mathbb{N}$.
Let $\alpha(\lambda,1)=1/\lambda$, $\lambda\in\mathbb{N}$.
Then $\beta(\lambda)=\lambda$ is unbounded.
However, $X_\lambda$, $\lambda\in\mathbb{N}$, is not an ARFS in $C_\infty[0,\infty)$ (see Section~\ref{SS:exponents_ARF}).
\end{remark}

We will prove a stronger result.
Theorem~\ref{T:exponents_ARFS} is a direct consequence of the following theorem (see Example~\ref{EX:dense}).
\vspace{1pt}
\begin{flushleft}
\textbf{Theorem~\ref{T:exponents_ARFS}$'$.}
\emph{Suppose that $\inf_{\lambda\in\Lambda}\alpha(\lambda,1)>0$.
If the family $\beta(\lambda)$, $\lambda\in\Lambda$, is unbounded, then $\bigcup_{\lambda\in\Lambda}X_\lambda$ is dense in $C_\infty[0,\infty)$.}
\end{flushleft}
\vspace{1pt}

To formulate our second result, we need a few auxiliary notions.
Let $n\in\mathbb{N}\cup\{\infty\}$, and let $\alpha_k$, $k\in I_n$, be positive numbers.
If $n\in\mathbb{N}$, then we assume that $\alpha_1<\ldots<\alpha_n$; if $n=\infty$, then we assume that $\alpha_1<\alpha_2<\ldots$.
Let $\Delta>0$.
We will say that the sequence $\alpha_k$, $k\in I_n$, satisfies the $\Delta$-gap condition if the following holds:
\begin{enumerate}
\item
if $n\in\mathbb{N}$ and $n\geqslant 2$, then $\alpha_{k+1}-\alpha_k\geqslant\Delta$, $k=1,\ldots,n-1$;
\item
if $n=\infty$, then $\alpha_{k+1}-\alpha_{k}\geqslant\Delta$, $k=1,2,\ldots$.
\end{enumerate}

\begin{theorem}\label{T:exponents_not_ARFS}
Suppose that there exists a $\Delta>0$ such that
the sequence $\alpha(\lambda,k)$, $k\in I_{n(\lambda)}$, satisfies the $\Delta$-gap condition for any $\lambda\in\Lambda$.
If the family $\beta(\lambda)$, $\lambda\in\Lambda$, is bounded, then $X_\lambda$, $\lambda\in\Lambda$, is not an ARFS in $C_\infty[0,\infty)$.
\end{theorem}

From Theorem~\ref{T:criterion_ARFS} it follows that $X_\lambda$, $\lambda\in\Lambda$, is not an ARFS in $C_\infty[0,\infty)$ iff
there exists a sequence $\varphi_n\in(C_\infty[0,\infty))^*$, $n\geqslant 1$, such that
\begin{equation}\label{E:condition_varphi}
\|\varphi_n\|=1,\,n\geqslant 1, \quad\text{and}\quad
\sup_{\lambda\in\Lambda}\|\varphi_n\upharpoonright_{X_\lambda}\|\to 0 \quad\text{as}\quad n\to\infty.
\end{equation}

Suppose
\begin{enumerate}
\item
there exists a $\Delta>0$ such that the sequence $\alpha(\lambda,k)$, $k\in I_{n(\lambda)}$, satisfies the $\Delta$-gap condition for any $\lambda\in\Lambda$;
\item
the family $\beta(\lambda)$, $\lambda\in\Lambda$, is bounded.
\end{enumerate}
We provide a "natural" sequence $\varphi_n$, $n\geqslant 1$, which satisfies~\eqref{E:condition_varphi}.
For $t\geqslant 0$, define $p_t:C_\infty[0,\infty)\to\mathbb{K}$ by
\begin{equation*}
p_t(f)=f(t),\quad f\in C_\infty[0,\infty).
\end{equation*}
Clearly, $p_t\in (C_\infty[0,\infty))^*$ and $\|p_t\|=1$, $t\geqslant 0$.
We claim that
\begin{equation*}
\sup_{\lambda\in\Lambda}\|p_t\upharpoonright_{X_\lambda}\|\to 0\quad\text{as}\quad t\to\infty.
\end{equation*}
This is a direct consequence of the following result (see Corollary~\ref{C:estimate_sup_p_t}).

\begin{theorem}\label{T:estimate_f(t)}
Let $n\in\mathbb{N}\cup\{\infty\}$, and let $\alpha_k$, $k\in I_n$, be positive numbers.
If $n\in\mathbb{N}$, then we assume that $\alpha_1<\ldots<\alpha_n$; if $n=\infty$, then we assume that $\alpha_1<\alpha_2<\ldots$.
Suppose $\Delta,M>0$ are such that
\begin{enumerate}
\item
the sequence $\alpha_k$, $k\in I_n$, satisfies the $\Delta$-gap condition;
\item
$\sum_{k\in I_n}1/\alpha_k\leqslant M$.
\end{enumerate}
Set $m=1/M$ and
\begin{equation*}
c=c(\Delta,M)=M\exp\left(2M+\frac{2}{M}+\frac{5\ln2}{2\Delta}+\frac{5\ln2}{\Delta M}+3\ln2+3\right).
\end{equation*}
Then
\begin{equation*}
|f(t)|\leqslant ce^{-mt}\|f\|,\quad t\geqslant 4M+\frac{5\ln2}{\Delta}+2,
\end{equation*}
for any $f\in Exp(\alpha_k\mid k\in I_n)$.
\end{theorem}

\begin{corollary}\label{C:estimate_sup_p_t}
Suppose $\Delta,M>0$ are such that
\begin{enumerate}
\item
the sequence $\alpha(\lambda,k)$, $k\in I_{n(\lambda)}$, satisfies the $\Delta$-gap condition for any $\lambda\in\Lambda$;
\item
$\beta(\lambda)\leqslant M$, $\lambda\in\Lambda$.
\end{enumerate}
Then
\begin{equation*}
\sup_{\lambda\in\Lambda}\|p_t\upharpoonright_{X_\lambda}\|\leqslant ce^{-mt},\quad t\geqslant  4M+\frac{5\ln2}{\Delta}+2,
\end{equation*}
where $m=1/M$, $c=c(\Delta,M)$.
\end{corollary}

\subsection{Proof of Theorem~\ref{T:exponents_ARFS}$'$}

The following lemma plays a crucial role in the proof of Theorem~\ref{T:exponents_ARFS}$'$.

\begin{lemma}\label{L:estimate_ARFS}
Let $N\in\mathbb{N}$, and let $\alpha,\alpha_1,\ldots,\alpha_N$ be distinct positive numbers.
Then
\begin{equation*}
d(e^{-\alpha t},\langle e^{-\alpha_k t}\mid k=1,\ldots,N\rangle)\leqslant\prod_{k=1}^N\left|1-\frac{\alpha}{\alpha_k}\right|,
\end{equation*}
where $\langle e^{-\alpha_k t}\mid k=1,\ldots,N\rangle$ is the subspace spanned by $e^{-\alpha_k t}$, $k=1,\ldots,N$.
\end{lemma}

This lemma is a direct consequence of the beautiful argument of M. von Golitschek~\cite[p.175, E1, a]{Borwein_Erdeliy}
(use the substitution $x=e^{-t}$).
For the convenience of the reader, we include its proof.

\begin{proof}[Proof of Lemma~\ref{L:estimate_ARFS}]
Define $f_0(t)=e^{-\alpha t}$ and
\begin{equation*}
f_k(t)=(\alpha_k-\alpha)\int_0^t e^{-\alpha_k(t-v)}f_{k-1}(v)\,dv,\quad k=1,\ldots,N.
\end{equation*}
By induction on $k$ it is easy to show that $f_k(t)=e^{-\alpha t}+g_k(t)$, where $g_k(t)\in\langle e^{-\alpha_j t}\mid j=1,\ldots,k\rangle$.
Hence, $f_N(t)=e^{-\alpha t}+g_N(t)$, where $g_N(t)\in\langle e^{-\alpha_j t}\mid j=1,\ldots,N\rangle$.

It is easily seen that $\|f_k\|\leqslant|1-\alpha/\alpha_k|\|f_{k-1}\|$, $k=1,\ldots,N$.
Since $\|f_0\|=1$, we conclude that $\|f_N\|\leqslant\prod_{k=1}^N|1-\alpha/\alpha_k|$.
This completes the proof.
\end{proof}

\begin{corollary}\label{C:estimate_ARFS}
Suppose that $\alpha_k\geqslant\alpha$, $k=1,\ldots,N$.
Using the inequality $1-x\leqslant e^{-x}$, $x\geqslant 0$, we get
\begin{equation*}
d(e^{-\alpha t},\langle e^{-\alpha_k t}\mid k=1,\ldots,N\rangle)\leqslant
\exp{\left(-\alpha\left(\frac{1}{\alpha_1}+\ldots+\frac{1}{\alpha_N}\right)\right)}.
\end{equation*}
\end{corollary}

Theorem~\ref{T:exponents_ARFS}$'$ follows from Corollary~\ref{C:estimate_ARFS} and the following well-known fact:
for any $\alpha_0>0$ the lineal spanned by $e^{-\alpha t}$, $\alpha\in(0,\alpha_0]$, is dense in $C_\infty[0,\infty)$.

\subsection{Proof of Theorem~\ref{T:estimate_f(t)}}

The following lemma plays a crucial role in the proof of Theorem~\ref{T:estimate_f(t)}.

\begin{lemma}\label{L:estimate_coef_f}
Let $N$ be a natural number, $\alpha_1,\ldots,\alpha_N$ positive numbers with $\alpha_1<\ldots<\alpha_N$.
Suppose that $\Delta,M>0$ are such that
\begin{enumerate}
\item
if $N\geqslant 2$, then $\alpha_{k+1}-\alpha_k\geqslant\Delta$, $k=1,\ldots,N-1$;
\item
$\sum_{k=1}^N 1/\alpha_k\leqslant M$.
\end{enumerate}
Set
\begin{equation*}
a=a(\Delta,M)=\exp(2M+\frac{5\ln 2}{2\Delta}+3\ln2),\quad b=b(\Delta,M)=4M+\frac{5\ln 2}{\Delta}+1.
\end{equation*}
If $f(t)=\sum_{k=1}^N a_ke^{-\alpha_k t}$, then
\begin{equation}\label{E:estimate_coefficients}
|a_k|\leqslant ae^{b\alpha_k}\|f\|,\quad k=1,\ldots,N.
\end{equation}
\end{lemma}

\begin{corollary}\label{C:estimate_f(t)}
Set $m=1/M$. We have
\begin{equation*}
|f(t)|\leqslant aMe^{m(b+1)-1}e^{-mt}\|f\|,\quad t\geqslant b+1.
\end{equation*}
\end{corollary}
\begin{proof}
Using~\eqref{E:estimate_coefficients}, we get
\begin{equation}\label{E:estimate_f_first_step}
|f(t)|\leqslant\sum_{k=1}^N|a_k|e^{-\alpha_k t}\leqslant a\|f\|\sum_{k=1}^N e^{b\alpha_k}e^{-\alpha_k t}=
a\|f\|\sum_{k=1}^Ne^{-\alpha_k}e^{\alpha_k(b+1-t)}.
\end{equation}
To estimate from above $e^{-\alpha_k}$, we will use the inequality $e^{-x}\leqslant 1/(ex)$, $x>0$.
Let us estimate from above $e^{\alpha_k(b+1-t)}$.
Since $\sum_{k=1}^N 1/\alpha_k\leqslant M$, we conclude that $\alpha_k\geqslant 1/M=m$, $k=1,\ldots,N$.
Suppose that $t\geqslant b+1$.
Then $\alpha_k(b+1-t)\leqslant m(b+1-t)$, $e^{\alpha_k(b+1-t)}\leqslant e^{m(b+1-t)}$.
Using~\eqref{E:estimate_f_first_step}, we have
\begin{equation*}
|f(t)|\leqslant a\|f\|\sum_{k=1}^N\frac{1}{e\alpha_k}e^{m(b+1-t)}\leqslant a\|f\|\frac{M}{e}e^{m(b+1-t)}=aMe^{m(b+1)-1}e^{-mt}\|f\|.
\end{equation*}
\end{proof}

Theorem~\ref{T:estimate_f(t)} is a direct consequence of Corollary~\ref{C:estimate_f(t)} (note that $c=aMe^{m(b+1)-1}$).

Now we proceed to the proof of Lemma~\ref{L:estimate_coef_f}.
We will need a few technical results. The following lemma is well-known (see, e.g., \cite[p.176-177, E2]{Borwein_Erdeliy}).

\begin{lemma}\label{L:distance}
Let $N$ be a natural number, $\gamma,\gamma_1,\ldots,\gamma_N$ distinct real numbers greater than $-1/2$.
Then
\begin{equation*}
d_{L_2[0,1]}(x^\gamma,\langle x^{\gamma_k}\mid k=1,\ldots,N\rangle)=\frac{1}{\sqrt{2\gamma+1}}\prod_{k=1}^N\frac{|\gamma-\gamma_k|}{\gamma+\gamma_k+1},
\end{equation*}
where $\langle x^{\gamma_k}\mid k=1,\ldots,N\rangle$ is the linear span of $x^{\gamma_k}$, $k=1,\ldots,N$.
\end{lemma}

Define
\begin{equation*}
\nu(x,y)=\frac{x+y}{|x-y|},\quad x,y\geqslant 0,\quad x\neq y.
\end{equation*}

The following lemma and its proof are motivated by~\cite[p.177, E3, a]{Borwein_Erdeliy}

\begin{lemma}\label{L:estimate_product_nu}
Let $x,y_1,\ldots,y_N>0$.
Suppose that $\Delta>0$ is such that $|x-y_k|\geqslant\Delta$, $k=1,\ldots,N$, and $|y_k-y_l|\geqslant\Delta$ for $k\neq l$.
Then
\begin{equation*}
\prod_{k=1}^N\nu(x,y_k)\leqslant \exp\left(\left(4\sum_{k=1}^N \frac{1}{y_k}+\frac{5\ln2}{\Delta}\right)x+3\ln2\right).
\end{equation*}
\end{lemma}
\begin{proof}
First, we will prove the required inequality for $\Delta=1$.
Define $\mathcal{K}_1=\{k\mid y_k<x\}$, $\mathcal{K}_2=\{k\mid y_k\in (x,2x)\}$, and $\mathcal{K}_3=\{k\mid y_k\geqslant 2x\}$.

Let us estimate $\prod_{k\in \mathcal{K}_1}\nu(x,y_k)$.
For $y\in[0,x)$ we have $\nu(x,y)=(x+y)/(x-y)$.
Hence, $\nu(x,y)$ increases in $y\in[0,x)$.
Set $m=[x]$.
We have
\begin{align*}
&\prod_{k\in \mathcal{K}_1}\nu(x,y_k)\leqslant\prod_{j=1}^m\nu(x,x-j)=\frac{\prod_{j=1}^m(2x-j)}{m!}\leqslant\\
&\leqslant\frac{\prod_{j=1}^m (2m+2-j)}{m!}=\binom{2m+1}{m}\leqslant 2^{2m+1}\leqslant 2^{2x+1}.
\end{align*}

Now we estimate $\prod_{k\in \mathcal{K}_2}\nu(x,y_k)$.
For $y>x$, we have $\nu(x,y)=(y+x)/(y-x)=1+2x/(y-x)$.
Hence, $\nu(x,y)$ decreases in $y\in(x,\infty)$.
We have
\begin{align*}
&\prod_{k\in \mathcal{K}_2}\nu(x,y_k)\leqslant\prod_{j=1}^m\nu(x,x+j)=\frac{\prod_{j=1}^m(2x+j)}{m!}\leqslant\\
&\leqslant\frac{\prod_{j=1}^m (2m+2+j)}{m!}=\binom{3m+2}{m}\leqslant 2^{3m+2}\leqslant 2^{3x+2}.
\end{align*}

Finally, we estimate from above $\prod_{k\in \mathcal{K}_3}\nu(x,y_k)$.
For $y>x$ we have
\begin{equation*}
\nu(x,y)=\frac{y+x}{y-x}=1+\frac{2x}{y-x}\leqslant\exp{\frac{2x}{y-x}}.
\end{equation*}
Since $2x/(y-x)\leqslant 4x/y$ for $y\geqslant 2x$, we conclude that $\nu(x,y)\leqslant\exp(4x/y)$, $y\geqslant 2x$.
It follows that
\begin{equation*}
\prod_{k\in \mathcal{K}_3}\nu(x,y_k)\leqslant\exp\left(4x\sum_{k=1}^N \frac{1}{y_k}\right).
\end{equation*}

Using the obtained estimates for $\prod_{k\in \mathcal{K}_l}\nu(x,y_k)$, $l=1,2,3$, we get
\begin{equation}\label{E:estimate_Delta_equals_1}
\prod_{k=1}^N\nu(x,y_k)\leqslant 2^{5x+3}\exp\left(4x\sum_{k=1}^N \frac{1}{y_k}\right)=
\exp\left(\left(4\sum_{k=1}^N \frac{1}{y_k}+5\ln 2\right)x+3\ln2\right).
\end{equation}

Let us prove the Lemma for any $\Delta>0$.
Using~\eqref{E:estimate_Delta_equals_1} for the numbers $x/\Delta,y_1/\Delta,\ldots,y_N/\Delta$, we get the required inequality.
\end{proof}

Now we are ready to prove Lemma~\ref{L:estimate_coef_f}.

\begin{proof}[Proof of Lemma~\ref{L:estimate_coef_f}]
If $N=1$, then the required assertion is obvious.
Let $N\geqslant 2$.
Define $g(x)=\sum_{k=1}^N a_k x^{\alpha_k}$, $x\in[0,1]$.
Then
\begin{equation*}
\|g\|_{L_2[0,1]}\leqslant\sup_{x\in[0,1]}|g(x)|=\sup_{t\in[0,\infty)}|g(e^{-t})|=\sup_{t\in[0,\infty)}|f(t)|=\|f\|.
\end{equation*}
Consider any $k\in\{1,\ldots,N\}$.
We have
\begin{equation*}
|a_k|d_{L_2[0,1]}(x^{\alpha_k},\langle x^{\alpha_j}\mid j\neq k\rangle)\leqslant\|g\|_{L_2[0,1]}\leqslant\|f\|,
\end{equation*}
hence,
\begin{equation*}
|a_k|\leqslant\frac{\|f\|}{d_{L_2[0,1]}(x^{\alpha_k},\langle x^{\alpha_j}\mid j\neq k\rangle)}.
\end{equation*}
Using Lemma~\ref{L:distance}, we get
\begin{equation*}
|a_k|\leqslant\sqrt{2\alpha_k+1}\prod_{j\neq k}\frac{\alpha_k+\alpha_j+1}{|\alpha_k-\alpha_j|}\|f\|=
\sqrt{2\alpha_k+1}\prod_{j\neq k}\nu\left(\alpha_k+\frac{1}{2},\alpha_j+\frac{1}{2}\right)\|f\|.
\end{equation*}
From Lemma~\ref{L:estimate_product_nu} it follows that
\begin{align*}
&|a_k|\leqslant\sqrt{2\alpha_k+1}
\exp\left(\left(4\sum_{j\neq k}\frac{1}{\alpha_j+\frac{1}{2}}+\frac{5\ln 2}{\Delta}\right)\left(\alpha_k+\frac{1}{2}\right)+3\ln2\right)\leqslant\\
&\leqslant\exp(\alpha_k)
\exp\left(\left(4M+\frac{5\ln 2}{\Delta}\right)\left(\alpha_k+\frac{1}{2}\right)+3\ln2\right)=ae^{b\alpha_k}.
\end{align*}
This completes the proof.
\end{proof}

\section{Acknowledgements}
The author is grateful to the referee of the previous version of this work for helpful suggestions, and
to A. V. Abanin for providing some of his papers.

\end{document}